\documentclass[preprint,review,10pt]{elsarticle}

\usepackage{amsmath,amssymb,amsfonts,amsthm}
\usepackage{hyperref}
\usepackage{graphicx}
\usepackage{mathrsfs}
\usepackage{caption}
\usepackage{float}
\usepackage{color}
\usepackage{cases}
\usepackage{bm}
\usepackage{pifont}
\newtheorem{assumption}{Assumption}

\newtheorem{lemma}{Lemma}

\newtheorem{theorem}{Theorem}

\let\mr=\mathrm
\usepackage{latexsym}
\usepackage{mathrsfs}
\usepackage{dsfont}
\usepackage{eufrak}
\usepackage{multirow}
 \usepackage{CJKutf8}
\usepackage{graphicx}
\usepackage{hyperref}
\usepackage{color}
\usepackage{diagbox}

\usepackage{enumerate}
\usepackage{cases}
\usepackage{xcolor}
\biboptions{sort&compress}
\begin{document}
\begin{CJK}{UTF8}{gbsn}
\begin{frontmatter}

\title{Supercloseness of the HDG method on Shishkin mesh for a singularly perturbed convection diffusion problem in 2D\tnoteref{funding} }

\tnotetext[funding]{
National Natural Science Foundation of China (11771257) and and Shandong Provincial Natural Science Foundation, China (ZR2023YQ002) support our paper.
}

\author[label1] {Xiaoqi Ma \fnref{cor1}}
\author[label1] {Jin Zhang \corref{cor2}}
\cortext[cor2] {Corresponding email: jinzhangalex@sdnu.edu.cn }
\fntext[cor1] {Address: xiaoqiMa@hotmail.com }
\address[label1]{School of Mathematics and Statistics, Shandong Normal University, Jinan 250014, China}

\begin{abstract}
In this paper, we present the first parameter uniform convergence analysis of a hybridizable discontinuous Galerkin (HDG) method based on a Shishkin mesh for a singularly perturbed convection-diffusion problem in 2D. {\color{red}The main challenge of this problem lies in the estimation of the convection term in the layer, the upper bound of which is often not given effectively by general methods. To solve this problem, a novel error control technique is used and reasonable assumptions are made for the stabilization function.} The results show that, when using polynomials with a degree not greater than $k$, the method achieves supercloseness of almost $k+\frac{1}{2}$ order in an energy norm, which is independent of the singular perturbation parameter.  Numerical experiment verifies the accuracy of theoretical analysis and confirms the efficiency of this method.
\end{abstract}

\end{frontmatter}
\section{Introduction}
The following singularly perturbed convection-diffusion problem is discussed:
\begin{equation}\label{eq:S-1}
\begin{aligned}
& Lu:=-\varepsilon\Delta u+\bm{\beta}\cdot\nabla u+cu=f, \quad \text{in $\Omega:= (0,1)^{2}$},\\
& u=0,\quad \text{on $\partial\Omega$},
\end{aligned}
\end{equation}
where $\varepsilon$ is called the perturbation parameter and satisfies $0<\varepsilon \ll 1$. For all $(x, y)\in\overline{\Omega}$, we make an assumption that $\bm{\beta}(x, y)\ge (\beta_{1}, \beta_{2})>(0, 0)$, $c(x, y)\ge 0$ and
\begin{equation}\label{eq:SPP-condition-1}
c(x, y)-\frac{1}{2}\nabla\cdot\bm{\beta}(x, y)\ge c_{0}>0, \quad\text{on $\overline{\Omega}$}
\end{equation}
with some fixed positive constants $\beta_{1}$, $\beta_{2}$ and $c_{0}$. The functions $\bm{\beta}$, $c$ and $f$ are sufficiently smooth. In fact, these assumptions ensures our model problem  has a unique solution in $H^{2}(\Omega)\cup H_{0}^{1} (\Omega)$ for all $f\in L^{2}(\Omega)$ \cite{Zhu1Zha2:2014--U, Zhu1Zha2:2013--C}. As $\varepsilon\rightarrow 0$, the exact solution $u$ of problem \eqref{eq:S-1} typically exhibits  boundary layers of width $\mathcal{O}(\varepsilon \ln (1/\varepsilon) )$ at $x= 1$ and $y=1$ and a corner layer at $(1, 1)$, see \cite{Roo1Sty2Tob3:2008--R}.  The presence of these layers poses significant challenges to traditional numerical methods. Firstly, traditional methods often struggle to strike a balance between computational efficiency and accuracy.  For example, the standard finite difference or finite element method cannot accurately capture the rapid changes in the solution without a sufficiently dense mesh to detail the layers. This insufficiency can lead to oscillations or instability in numerical solutions, especially when higher-order derivatives are required or when solving on complex geometries. Secondly, most traditional methods are not sensitive enough to the parameters to provide the so-called ``parameter uniform", that is, the quality of the numerical solution should be independent of the size of the perturbation parameter. In practical applications, this limitation reduces the predictive power and reliability of the model, especially in the case of a wide range of parameters.

To address these challenges, Shishkin meshes \cite{shishkin1992discrete} were introduced. These meshes effectively capture key properties of solutions by using finer meshes in areas requiring high resolution, thereby improving the accuracy of numerical solutions. Furthermore, Shishkin meshes have demonstrated the ability to maintain parameter uniform in dealing with singularly perturbed problems \cite{Zhu1Zha2:2014--U, Zhu1Zha2:2013--C, Mil1ORi2Shi3:1996--F, Roo1Zar2:2007--motified}, as they can automatically adjust the mesh structure according to the perturbation parameters to accommodate the sensitivity of the solution.

In this paper, we focus on the hybridizable discontinuous Galerkin (HDG) method \cite{C0c1Gop2Laz3:2009--U}, which inherits and develops many advantages of the discontinuous Galerkin (DG) method, such as its applicability to various partial differential equations, its ability to handle complex geometry and support high order accuary, see \cite{Ma1Zha2:2023--S, Buf1Hug2San3:2006--A, Coc1Shu2:1998--motified, Ku1:2014--O, Lin1Ye2Zha3:2018--motified, Liu1Zha2:2017--A, Xie1Zha2Zha4, Zhu1Tia2Zha3:2011--C, Xie1Zha2:2007--S} for more details. The core advantage of the HDG method lies in expressing the approximate scalar variable and flux element-wise using the approximate trace along the element boundaries, and by enforcing flux continuity to obtain the unique trace at element boundaries. Therefore, compared with the DG method, the HDG method simplifies the global equation system into a system that only involves approximate trace on the boundaries, greatly improving the computational efficiency \cite{Coc1:2016--S}. Based on these advantages, the HDG method has achieved remarkable success in solving various problems, such as Poisson equation \cite{Coc1Gop2Say3:2010--motified, Kir1She2Coc3:2012--T}, convection-diffusion equation \cite{Ngu1Per2Coc3:2009--motified, Fu1Oiu2Zha3:2015--motified}, Stokes equation \cite{Coc1Gop2:2009--motified, Ngu1Per2Coc3:2010--motified}, Navier-Stokes equation \cite{Ngu1Per2Coc3:2011--motified, Ces1Coc2Qiu3:2017--A}, Maxwell equation \cite{Ngu1Per2Coc3:2011--H, Li1Lan2Per3:2014--motified} and others.

{\color{red}Although the HDG method has been widely used in many fields and extensively discussed and improved in several works \cite{Bus1Lom2Sol3:2019--motified, Che1Coc2:2012--A, Coc1Don2Gum3:2008--motified, Coc1Don2:2009--motified, Li1Wan2:2022--H, Qiu1Shi2:2016--motified}, its \emph{a priori} error analysis for convection-dominated diffusion problems was not introduced until 2016 by G. Fu et al \cite{Fu1Oiu2Zha3:2015--motified}. They demonstrated that on general conforming quasi-uniform simplicial meshes, using no more than $k$-degree polynomials, the $L^2$ error for the scalar variable converges with order $k+\frac{1}{2}$. R. Bustinza et al. later extended the use of HDG method on anisotropic triangulations for a convection-dominated diffusion problem and, under certain assumptions of stabilization parameter and mesh, proved that all variables can recover to $k+\frac{1}{2}$ in the $L^2$ norm \cite{Bus1Lom2Sol3:2019--motified}. Despite these significant advancements in the theoretical understanding of the HDG method, none of these studies addressed parameter uniform convergence---a crucial metric for assessing the performance of numerical methods under extreme parameter conditions. Recently, Y. Li et al. \cite{Li1Wan2:2022--H} proposed a parameter uniform HDG method on Shishkin meshes for a singularly perturbed problem and conducted preliminary numerical experiments. While these experimental results demonstrate the effectiveness of the method, corresponding theoretical convergence analysis has yet to be undertaken, indicating that the theoretical foundation of the HDG method regarding parameter uniform still requires further exploration.}

In this paper, we study the parameter uniform convergence of  the HDG method on a Shihskin mesh for a singularly perturbed convection-diffusion problem in 2D. Given the theoretical challenges posed by the convergence analysis of the convection term in the layers, we use an innovative error control technique introduced in \citep[Lemma 3.1]{Wan1Wan2Zha3:2016--L}. This technique controls the derivative of $\xi_{u}:=u_{h}-\Pi_{2}u$ through the energy norm itself (see Lemma \ref{L-1}). Here $u_{h}$ and $\Pi_{2}u$ denote the numerical solution and projection of the exact solution $u$ of the problem \eqref{eq:S-1}, respectively. Despite a negative power of the small parameter still appearing in the upper bound, we can improve the error estimate of the convection term by utilizing small measures of the layer region. Utilizing this strategy and under certain reasonable assumptions about the stabilization function, the superconvergence of almost $k+1/2$ in the energy norm is proved, where $k$ represents the maximum degree of polynomials used to approximate the scalar variable, the flux, and the trace of the scalar variable. 
Numerical experiment further verify the reliability of our theoretical results.

This paper is structured as follows: In Section 2, we introduce \emph{a prior} information of the solution. Then, a HDG method on a Shishkin mesh for this problem is established. Section 3 defines some local projection operators and obtains projection error estimates. Moreover,  an error handling technique is proposed. In Section 4, the supercloseness between the projection and numerical solution in an energy norm is derived. Finally, some numerical results are dedicated to confirm the main theoretical conclusion.

In this paper, a generic positive constant $C$ is independent of $\varepsilon$ and the number of mesh points $N$. Moreover, $k$ is some fixed integer satisfying the condition $k\ge 1$.
\section{\emph{A priori} information, Shishkin mesh and HDG method}\label{sec:mesh,method}
\subsection{\emph{A priori} information}
\begin{theorem}
We decompose the exact solution $u$ of problem \eqref{eq:S-1}  into $u = S+E_{1}+E_{2}+E_{3}$, where for all $(x, y)\in \Omega$ the component functions have the regularity
\begin{equation}\label{eq:decomposition}
\begin{aligned}
& \left|\frac{\partial^{i+j}S}{\partial^{i}x\partial^{j}y}(x, y)\right|\le C,\\
& \left|\frac{\partial^{i+j}E_{1}}{\partial^{i}x\partial^{j}y}(x, y)\right|\le C\varepsilon^{-i}e^{-\beta_{1}(1-x)/\varepsilon},\\
&\left|\frac{\partial^{i+j}E_{2}}{\partial^{i}x\partial^{j}y}(x, y)\right|\le C\varepsilon^{-j}e^{-\beta_{2}(1-y)/\varepsilon},\\
&\left|\frac{\partial^{i+j}E_{3}}{\partial^{i}x\partial^{j}y}(x, y)\right|\le C\varepsilon^{-(i+j)}e^{-[\beta_{1}(1-x)+\beta_{2}(1-y)]/\varepsilon},
\end{aligned}
\end{equation}
for $0\le i+j\le k+2$. Here constant C depends on $\bm{\beta}$, $c$ and $f$.
\end{theorem}
\begin{proof}
By means of the arguments in \cite{Roo1Sty2Tob3:2008--R}, we can draw this conclusion immediately. 
\end{proof}

\subsection{Shishkin mesh}
Below we will introduce a piecewise uniform---Shishkin mesh, which are refined near the sides $x = 1$ and $y = 1$ of $\Omega$. 

Let's present some basic concepts. The set of mesh points are defined by $\Omega_{N}=\{(x_{i}, y_{j})\in \overline{\Omega} : i, j=0, 1, \cdots, N\}$. 
A mesh rectangle is often written as  $K_{ij}=I_{i}\times J_{j} = (x_{i-1}, x_{i})\times (y_{j-1}, y_{j})$ for $1\le i, j\le N$.
Furthermore, $h_{i, x}=x_{i}-x_{i-1}$ and $h_{j, y}=y_{j}-y_{j-1}$ denote the length of $I_{i}$ and $J_{j}$, respectively. 

In order to distinguish the layer part from the smooth part, we introduce two transition points $\tau_{x}$ and $\tau_{y}$, where
$$\tau_{x}=\min\left\{\frac{1}{2}, \frac{\sigma_{x} \varepsilon}{\beta_{1}} \ln N\right\},\quad \tau_{y}=\min\left\{\frac{1}{2}, \frac{\sigma_{y} \varepsilon}{\beta_{2}} \ln N\right\}.$$
Here $\beta_{1}$ and $\beta_{2}$ are the lower bounds for the convective part $\bm{\beta}$. For convenience, we select $\sigma_{x}=\sigma_{y}=\sigma$ and $\sigma\ge k+1$. Suppose that $N\in \mathbb{N}$ can be divided by $4$ and $N\ge 4$. $[0, 1-\tau_{x}]$ and $[1-\tau_{x}, 1]$ are each divided into $N/2$ equidistant mesh intervals, and the division of $y$-direction is similar. Shishkin mesh points can be denoted as
\begin{equation*}
x_{i}=
\left\{
\begin{aligned}
&\frac{2(1-\tau_{x})}{N}i\quad &&\text{for $i=0, 1, \cdots, N/2$},\\
&1-\tau_{x}+\frac{2\tau_{x}}{N}(i-\frac{N}{2}),&&\text{for $i=N/2+1, \cdots, N$}
\end{aligned}
\right.
\end{equation*}
and 
\begin{equation*}
y_{j}=
\left\{
\begin{aligned}
&\frac{2(1-\tau_{y})}{N}j\quad &&\text{for $j=0, 1, \cdots, N/2$},\\
&1-\tau_{y}+\frac{2\tau_{y}}{N}(j-\frac{N}{2}),&&\text{for $j=N/2+1, \cdots, N$}.
\end{aligned}
\right.
\end{equation*}

Now the domain $\Omega$ is divided into four parts, as shown in Figure \ref{WW-2}:
\begin{equation*}
\begin{aligned}
&\Omega_{s}=(0, 1-\tau_{x})\times(0, 1-\tau_{y}),\quad \Omega_{x}=(1-\tau_{x}, 1)\times(0, 1-\tau_{y}),\\
&\Omega_{y}=(0, 1-\tau_{x})\times(1-\tau_{y}, 1),\quad\Omega_{xy}=(1-\tau_{x}, 1)\times (1-\tau_{y}, 1).
\end{aligned}
\end{equation*}
According to the above definition, there is $x_{N/2}=1-\tau_{x}$ and $y_{N/2}=1-\tau_{y}$.

\begin{assumption}\label{ass:S-1}
In this paper, we shall suppose that
\begin{equation*}
\varepsilon \le C N^{-1},
\end{equation*}
which is not a limitation in our practice life.
\end{assumption}

\subsection{The HDG method}
Let $\mathcal{T}_{h}:=\{K_{ij}\}_{i, j=1, \cdots, N}$, where each rectangular mesh element $K_{ij}:= I_{i}\times J_{j}:= (x_{i-1}, x_{i})\times (y_{j-1}, y_{j})$. Suppose that $K\in\mathcal{T}_{h}$ is a general open set, $\partial K$ represents the set of edges of $K$, and $\partial\mathcal{T}_{h}:=\{\partial K: K\in \mathcal{T}_{h}\}$.

 Then we use $\mathcal{E}_{h}^{0}$ to denote the set of interior edges and $\mathcal{E}^{\alpha}_{h}$ to denote the set of edges on the boundary. If there are two elements $K^{+}$ and $K^{-}$ in $\mathcal{T}_{h}$ such that $e=\partial K^{+}\cap \partial K^{-}$, we call $e\in \mathcal{E}_{h}^{0}$.
If there is an element $K$ in $\mathcal{T}_{h}$ that causes $e=\partial K\cap \partial\Omega$, we call $e \in \mathcal{E}_{h}^{\alpha}$.
Here $\mathcal{E}_{h}=\mathcal{E}_{h}^{0}\cup\mathcal{E}_{h}^{\alpha}$.

The standard Sobolev spaces $W^{m,l}(D)$ and $L^{l}(D)$ will be used, where $D$ is any measurable one-dimensional subset of $[0, 1]$ or any measurable two-dimensional subset of $\Omega$. The $L^{2}(D)$ norm is denoted by $\Vert \cdot\Vert_{D}$, the $L^{\infty}(D)$ norm by $\Vert\cdot\Vert_{L^{\infty}(D)}$, and $(\cdot, \cdot)_{D}$ denotes the $L^{2}(D)$ inner product. The subscript $D$ will always be omitted when $D = \Omega$.

 Define inner product
\begin{equation*}
\begin{aligned}
&(\eta, w)_{\mathcal{T}_{h}}:=\sum_{K\in\mathcal{T}_{h}}\int_{K}\eta w,\\
&<z, \omega>_{\partial\mathcal{T}_{h}}:=\sum_{K\in \mathcal{T}_{h}}\int_{\partial K}z\omega\mr{d}s
\end{aligned}
\end{equation*}
for any function 
$\eta, w\in L^{2}(\Omega)$
and $z, \omega\in L^{2}(\partial\mathcal{T}_{h})$. The definition of $(\cdot, \cdot)_{\mathcal{T}_{h}}$ for vector functions can be given in the same way. 
Then we define the finite element space by
\begin{equation*}
\begin{aligned}
&\bm{V}_{h}:=\{\bm{v}\in [L^{2}(\Omega)]^{2}: \bm{v}|_{K}\in [\mathcal{Q}^{k}(K)]^{2},\quad \forall K\in \mathcal{T}_{h}\},\\
&W_{h}:=\{v\in L^{2}(\Omega), w|_{K}\in \mathcal{Q}^{k}(K), \quad \forall K\in \mathcal{T}_{h}\},\\
&M_{h}:=\{\mu\in L^{2}(\mathcal{E}_{h}), \mu|_{e}\in \mathbb{P}^{k}(e),\quad \forall e\in \mathcal{E}_{h}\},\\
&M_{h}(0):=\{\mu\in M_{h}: <\mu,\xi>_{\partial\Omega}=0,\quad\forall \xi\in M_{h}\},
\end{aligned}
\end{equation*}
where $\mathcal{Q}^{k}(K)$ is the tensor products polynomial space of degree at most $k$ in one variable with $k\ge 1$, and $\mathbb{P}^{k}(e)$ is the set of polynomials defined on edges $e$ with the highest degree not exceeding $k$.

Now rewrite the problem \eqref{eq:S-1} into the following first-order system using $q=-\varepsilon \nabla u$
\begin{equation*}
\left\{
\begin{aligned}
&\bm{q}+\varepsilon \nabla u=0,\quad&& \text{in $\Omega$},\\
&\nabla\cdot \bm{q}+\bm{\beta}\cdot\nabla u+cu=f\quad && \text{in $\Omega$},\\
&u=0\quad &&\text{on $\partial\Omega$}.
\end{aligned}
\right.
\end{equation*}
The numerical approximation solution $(\bm{q}_{h}, u_{h}, \hat{u}_{h})\in \bm{V}_{h}\times W_{h}\times M_{h}(0)$ can be defined as a solution of the system
\begin{equation*}
\left\{
\begin{aligned}
&\varepsilon^{-1}(\bm{q}_{h}, \bm{r})_{\mathcal{T}_{h}}-(u_{h}, \nabla\cdot \bm{r})_{\mathcal{T}_{h}}+<\hat{u}_{h}, \bm{r}\cdot \bm{n}>_{\partial\mathcal{T}_{h}}=0\\
&-(\bm{q}_{h}+\bm{\beta}u_{h}, \nabla w)_{\mathcal{T}_{h}}+((c-\nabla\cdot \bm{\beta})u_{h}, w)_{\mathcal{T}_{h}}+<(\widehat{\bm{q}_{h}}+\widehat{\bm{\beta}u_{h}})\cdot \bm{n}, w>_{\partial\mathcal{T}_{h}}=(f, w)_{\mathcal{T}_{h}}\\
&<(\widehat{\bm{q}_{h}}+\widehat{\bm{\beta}u_{h}})\cdot \bm{n}, \mu>_{\partial\mathcal{T}_{h}/\partial\Omega}=0,\\
\end{aligned}
\right.
\end{equation*}
for $(\bm{r}, w, \mu)\in \bm{V}_{h}\times W_{h}\times M_{h}(0)$, where
\begin{equation*}
\begin{aligned}
(\widehat{\bm{q}_{h}}+\widehat{\bm{\beta}u_{h}})\cdot\bm{n}:=\bm{q}_{h}\cdot\bm{n}+\bm{\beta}\cdot \bm{n}\hat{u}_{h}+\tau(u_{h}-\hat{u}_{h}),\quad \text{on $\partial\mathcal{T}_{h}$}.
\end{aligned}
\end{equation*}
Here the numerical trace $\hat{u}_{h}$ is an approximation of $u$ on $\mathcal{E}_{h}$, $\tau$ is a piecewise nonnegative stabilization function defined on $\partial\mathcal{T}_{h}$, and $\bm{n}$ is the unit outward normal vector on $\partial K$. In this paper, the stabilization function $\tau$ is  taken as a constant that satisfy
\begin{equation*}
\begin{aligned}
\tau-\frac{1}{2}\bm{\beta}\cdot \bm{n}> 0,\quad \text{on $\partial\mathcal{T}_{h}$}.
\end{aligned}
\end{equation*}
For the subsequent convergence analysis, we consider that $\tau$ satisfies the following properties:
\begin{enumerate}[(1)]
\item $\exists C_{0} > 0$ such that $\max\limits_{e\in \partial K} \tau|e =: \tau^{max} \le C_{0},\quad \forall K\in \mathcal{T}_{h}$.
\item $\exists C_{1} > 0$ such that $\tau^{v}:=\min\limits_{e\in \partial K}\inf\limits_{e}(\tau-\frac{1}{2}\bm{\beta}\cdot \bm{n})\ge C_{1},\quad \forall K\in \mathcal{T}_{h}$.
\item $\exists C_{2} > 0$ such that $\inf\limits_{e}(\tau-\frac{1}{2}\bm{\beta}\cdot \bm{n})\ge C_{2}\max\limits_{e} |\bm{\beta}\cdot \bm{n}|$ for all $e\in \partial K, K\in \mathcal{T}_{h}$.
\end{enumerate}


Then, the weak formula of the problem \eqref{eq:S-1} is that: Find $(\bm{q}_{h}, u_{h}, \hat{u}_{h})\in \bm{V}_{h}\times W_{h}\times M_{h}(0)$ such that
\begin{equation}\label{eq:SD}
B((\bm{q}_{h}, u_{h}, \hat{u}_{h}); (\bm{r}, w, \mu))=F(w) \quad \forall (\bm{r}, w, \mu) \in \bm{V}_{h}\times W_{h}\times M_{h}(0),
\end{equation}
where
\begin{equation*}
\begin{aligned}
B((\bm{q}_{h}, u_{h}, \hat{u}_{h}); (\bm{r}, w, \mu))&=\varepsilon^{-1}(\bm{q}_{h}, \bm{r})_{\mathcal{T}_{h}}-(u_{h}, \nabla\cdot \bm{r})_{\mathcal{T}_{h}}+<\hat{u}_{h}, \bm{r}\cdot \bm{n}>_{\partial\mathcal{T}_{h}}\\
&-(\bm{q}_{h}+\bm{\beta}u_{h}, \nabla w)_{\mathcal{T}_{h}}+((c-\nabla\cdot \bm{\beta})u_{h}, w)_{\mathcal{T}_{h}}\\
&+<(\widehat{\bm{q}_{h}}+\widehat{\bm{\beta}u_{h}})\cdot \bm{n}, w-\mu>_{\partial\mathcal{T}_{h}}
\end{aligned}
\end{equation*}
and $F(w)=(f, w)_{\mathcal{T}_{h}}.$ 

\begin{lemma}\label{Galerkin orthogonality property}
Let $(\bm{q}, u)$ be the solution of \eqref{eq:S-1}, then the following Galerkin orthogonality property
\begin{equation*}
B((\bm{q}-\bm{q}_{h}, u-u_{h}, u-\hat{u}_{h}), (\bm{r}, w, \mu))=0
\end{equation*}
holds for all $(\bm{r}, w, \mu) \in \bm{V}_{h}\times W_{h}\times M_{h}(0)$, where $B(\cdot, \cdot)$ is the bilinear form defined in \eqref{eq:SD}.
\end{lemma}
\begin{proof}
Through some simple calculations, this lemma can be easily derived, see \citep[Section 4.2]{Fu1Oiu2Zha3:2015--motified}.
\end{proof}

The energy norm related to $B(\cdot, \cdot)$ is defined as: for all $(\bm{r}, w, \mu) \in \bm{V}_{h}\times W_{h}\times M_{h}(0)$, there is
\begin{equation}\label{eq:SS-1}
\begin{aligned}
|||(\bm{r}, w, \mu)|||^{2}:=
\varepsilon^{-1}\Vert \bm{r}\Vert_{\mathcal{T}_{h}}^{2}+\Vert (c-\frac{1}{2}\nabla\cdot\bm{\beta})^{\frac{1}{2}}w\Vert_{\mathcal{T}_{h}}^{2}+ \Vert(\tau-\frac{1}{2}\bm{\beta}\cdot \bm{n})^{\frac{1}{2}}(w-\mu)\Vert^{2}_{\partial\mathcal{T}_{h}}.
\end{aligned}
\end{equation}
And from \eqref{eq:SPP-condition-1}, the bilinear form $B(\cdot, \cdot)$ is coercive, i.e., 
\begin{equation}\label{eq:coercity}
B((\bm{r}, w, \mu), (\bm{r}, w, \mu)) \ge |||(\bm{r}, w, \mu)|||^{2},\quad \forall (\bm{r}, w, \mu)\in \bm{V}_{h}\times W_{h}\times M_{h}(0).
\end{equation}

%
%
%

\section{Projection and projection errors}

\subsection{Projection}
In this paper, for $\bm{q}$, $u$ and $\hat{u}_{h}$ we introduce the standard local $L^{2}$-projection operators $\bm{\Pi_{1}}$, $\Pi_{2}$ and $P$ onto $\bm{V}_{h}$, $W_{h}$ and $M_{h}$, respectively, see \citep[Section 3]{Qiu1Shi2:2016--motified} for more details.

For any $u\in L^{2}(K)$, we define $\Pi_{2}u \in\mathcal{Q}^{k}(K)$ as follows:
\begin{align}\label{eq:J-1}
&\int_{K}(\Pi_{2}u)v\mr{d}x\mr{d}y=\int_{K}uv\mr{d}x\mr{d}y,\quad \forall v\in \mathcal{Q}^{k}(K),
\end{align}
where $\mathcal{Q}^{k}(K)$ is the tensor products polynomial space of degree at most $k$ in one variable with $k\ge 1$. The local $L^2$ projection $\bm{\Pi_{1}}\bm{q}$ for the vector function $\bm{q}$ is also given in the same way.

Moreover, we will present the local $L^{2}$ projection operator $P$ from $L^{2}(\mathcal{E}_{h})$ onto $M_{h}$ (see \citep[Section 5]{Che1:2021--O}):
For any $z\in L^{2}(I_{i})$, $Pz\in \mathbb{P}^{k}(I_{i})$ is defined by
\begin{align}\label{eq:J-1}
&\int_{I_i}(Pz)v\mr{d}x=\int_{I_i}zv\mr{d}x,\quad \forall v\in \mathbb{P}^{k}(I_{i}),
\end{align}
for any $I_{i}=(x_{i-1}, x_{i}), i=1, 2,\cdots, N$. This is also defined for $J_{j}=(y_{j-1}, y_{j}), j=1, 2, \cdots, N$.

\subsection{Projection errors}
Assume that $\Pi_{2}$ is the local $L^{2}$-projection operator and $h = diam(K)$ is the diameter of $K$. Then we have the following approximation property, see \citep[Theorem 1]{Cro18jTho2:1987--motified} and \citep[Lemma 1]{Che1Pi2Xu3:2019--motified}:
\begin{equation}\label{P-111}
\Vert w- \Pi_{l} w\Vert_{K} + h^{\frac{1}{2}}\Vert w-\Pi_{l} w\Vert_{\partial K} \le Ch^{k+1}|w|_{H^{k+1}(K)},\quad K\in\mathcal{T}_{h}
\end{equation}
for all $w\in H^{k+1}(K)$. Similarly, it is noted that this estimate also holds for $\bm{\Pi_{1}}\bm{w}$.
For the operator $P$, according to the projection results \citep[Lemma 5.1]{Che1:2021--O}, we have
\begin{align*}
&\Vert z-Pz \Vert_{L^{m}(I_{i})}\le C h_{i,x}^{k+1}\Vert z^{(k+1)} \Vert_{L^{m}(I_{i})}, \quad m=2,\infty \label{eq:interpolation-theory-3}
\end{align*}
for all $z\in H^{k+1}(I_{i})$, where $I_{i}=(x_{i-1},x_{i})$ and $h_{i,x}$​​ is the length of $I_{i}$​​.  These results also apply to $J_{j}, j=1, 2, \cdots, N$.

Set $\bm{e}:=(\bm{e_{q}}, e_{u}, \hat{e}_{u}):=(\bm{q}-\bm{q}_{h}, u-u_{h}, u-\hat{u}_{h})$, which can also write as
\begin{equation}\label{PPP-1}
\bm{e}=\bm{w}-\bm{w}_{h}=(\bm{w}-\Pi\bm{w})-(\bm{w}_{h}-\Pi\bm{w})=\bm{\eta}-\bm{\xi},
\end{equation}
where
\begin{equation*}
\begin{aligned}
&\bm{w}=(\bm{q}|_{\Omega}, u|_{\Omega}, u|_{\mathcal{E}_{h}}),\\
&\bm{w}_{h}=(\bm{q}_{h}, u_{h}, \hat{u}_{h}),\\
&\bm{\eta}=(\bm{\eta_{q}}, \eta_{u}, \hat{\eta}_{u})=(\bm{q}-\bm{\Pi_{1}q}, u-\Pi_{2} u, u-Pu),\\
&\bm{\xi}=(\bm{\xi_{q}}, \xi_{u}, \hat{\xi}_{u})=(\bm{q}_{h}-\bm{\Pi_{1}q}, u_{h}-\Pi_{2}u, \hat{u}_{h}-Pu),
\end{aligned}
\end{equation*}
and here we define $\Pi\bm{w}=(\bm{\Pi_{1}q}, \Pi_{2}u, Pu)$.  And on the basis of that, it is straightforward to obtain the following theorem.

\begin{theorem}\label{error}
Let Assumption \ref{ass:S-1} and $\sigma\ge k+1$ hold. Then on the Shishkin mesh, there is
\begin{equation*}
|||\bm{\eta}|||\le CN^{-(k+\frac{1}{2})}(\ln N)^{k+1},
\end{equation*}
where $\bm{\eta}=(\bm{q}-\bm{\Pi_{1}q}, u-\Pi_{2} u, u-Pu)$, $(\bm{q}, u)$ is the exact solution of the problem \eqref{eq:S-1}, and $(\bm{\Pi_{1}q}, \Pi_{2}u, Pu)$ is the  projection of the exact solution of \eqref{eq:S-1}.  
\end{theorem}
\begin{proof}
According to the definition of the norm \eqref{eq:SS-1},
\begin{equation*}
\begin{aligned}
|||\bm{\eta}|||^{2}:&=\varepsilon^{-1}\Vert\bm{\eta_{q}}\Vert_{\mathcal{T}_{h}}^{2}+\Vert (c-\frac{1}{2}\nabla\cdot\bm{\beta})^{\frac{1}{2}}\eta_{u}\Vert^{2}_{\mathcal{T}_{h}}+\Vert(\tau-\frac{1}{2}\bm{\beta}\cdot \bm{n})^{\frac{1}{2}}(\eta_{u}-\hat{\eta}_{u})\Vert^{2}_{\partial\mathcal{T}_{h}}\\
&:=\mathcal{Y}_{1}+\mathcal{Y}_{2}+\mathcal{Y}_{3}.
\end{aligned}
\end{equation*}

First, for $\mathcal{Y}_{1}$, using Theorem \ref{eq:decomposition} and the interpolation error \eqref{P-111}, we have
\begin{equation*}
\mathcal{Y}_{1}\le C\varepsilon^{-1}\varepsilon N^{-2(k+1)}(\ln N)^{2k+3}\le CN^{-2(k+1)}(\ln N)^{2k+3}.
\end{equation*}

For $\mathcal{Y}_{2}$, by using the same method, 
$$\mathcal{Y}_{2}\le C\Vert\eta_{u}\Vert^{2}_{\mathcal{T}_{h}}\le C\left(N^{-2(k+1)}+\varepsilon N^{-2(k+1)}(\ln N)^{2k+3}\right).$$

For $\mathcal{Y}_{3}$, it's easy to derive
\begin{equation*}
\begin{aligned}
\mathcal{Y}_{3}\le C\left(N\Vert\eta_{u}\Vert^{2}_{L^{\infty}(\Omega)}+\Vert\hat{\eta}_{u}\Vert^{2}_{\partial\mathcal{T}_{h}}\right)\le CN^{-(2k+1)}(\ln N)^{2(k+1)}.
\end{aligned}
\end{equation*}
Thus, this proof is completed.
\end{proof}

In order to obtain the parameter uniform supercloseness analysis of the HDG method on a Shishkin mesh, we introduce the following new error handling technique \citep[Lemma 3.3]{Che1Jia2Sty3:2023--S}.

\begin{lemma}\label{L-1}
There is a constant $C$ that is independent of the perturbation parameter $\varepsilon$ and the mesh parameter $N$ such that
\begin{equation*}
\begin{aligned}
&\Vert(\xi_{u})_{x}\Vert_{\Omega_{x}}\le Ch_{i, x}^{-\frac{1}{2}}|||\bm{\xi}|||,\\
&\Vert(\xi_{u})_{y}\Vert_{\Omega_{y}}\le Ch_{j, y}^{-\frac{1}{2}}|||\bm{\xi}|||,
\end{aligned}
\end{equation*}
where $\xi_{u}:=u_{h}-\Pi_{2}u$, $\Omega_{x}=(1-\tau_{x}, 1)\times(0, 1-\tau_{y})$ and $\Omega_{y}=(0, 1-\tau_{x})\times(1-\tau_{y}, 1)$. Moreover, $h_{i, x}=x_{i}-x_{i-1}$ and $h_{j, y}=y_{j}-y_{j-1}$ denote the length of $I_{i}$ and $J_{j}$, respectively.
\end{lemma}
\begin{proof}
Below, we only prove the estimates hold on $\Omega_{x}$, and the similar estimate on $\Omega_{y}$ can be obtained using the same method. Set $w=\mu=0$ in the test function $(\bm{r}, w, \mu)$ in bilinear form \eqref{eq:SD}, then write the local form as
$$\varepsilon^{-1}(\bm{q}_{h}, \bm{r})_{K_{ij}}-(u_{h}, \nabla\cdot\bm{r})_{K_{ij}}+<\hat{u}_{h}, \bm{r}\cdot \bm{n}>_{\partial K_{ij}}=0$$
Acoording to $\bm{r}=(r^{1}, r^{2})$, $\bm{q}_{h}=(q^{1}_{h}, q^{2}_{h})$ and $\bm{n}=(n_{x}, n_{y})$, suppose that $r^{2}=0$, one has
\begin{equation*}
\begin{aligned}
&\varepsilon^{-1}(q^{1}_{h}, r^{1})_{K_{ij}}-(u_{h}, r^{1}_{x})_{K_{ij}}+\left<\hat{u}_{h}(x_{i}^{-}, y), r^{1}(x_{i}^{-}, y)\right>_{J_{j}}\\
&-\left<\hat{u}_{h}(x_{i-1}^{+}, y), r^{1}(x_{i-1}^{+}, y)\right>_{J_{j}}=0.
\end{aligned}
\end{equation*}
Here $u$ and $q^{1}=-\varepsilon u_{x}$ also satisfy the above formula, then one has
\begin{equation*}
\begin{aligned}
&\varepsilon^{-1}(q^{1}-q^{1}_{h}, r^{1})_{K_{ij}}-(u-u_{h},  r^{1}_{x})_{K_{ij}}\\
&+\left<(u-\hat{u}_{h})(x_{i}^{-}, y), r^{1}(x_{i}^{-}, y)\right>_{J_{j}}-\left<(u-\hat{u}_{h})(x_{i-1}^{+}, y), r^{1}(x_{i-1}^{+}, y)\right>_{J_{j}}=0.
\end{aligned}
\end{equation*}
Using the error equation \eqref{PPP-1} we have
\begin{equation*}
\begin{aligned}
&\varepsilon^{-1}(\eta_{q^{1}}, r^{1})_{K_{ij}}-\varepsilon^{-1}(\xi_{q^{1}}, r^{1})_{K_{ij}}-(\eta_{u}, r^{1}_{x})_{K_{ij}}+(\xi_{u}, r^{1}_{x})_{K_{ij}}\\&+<\eta_{\hat{u}}(x_{i}^{-}, y), r^{1}(x_{i}^{-}, y)>_{J_{j}}-<\xi_{\hat{u}}(x_{i}^{-}, y), r^{1}(x_{i}^{-}, y)>_{J_{j}}\\
&-<\eta_{\hat{u}}(x_{i-1}^{+}, y), r^{1}(x_{i-1}^{+}, y)>_{J_{j}}+<\xi_{\hat{u}}(x_{i-1}^{+}, y), r^{1}(x_{i-1}^{+}, y)>_{J_{j}}=0
\end{aligned}
\end{equation*}
Due to the use of local $L^{2}$ projection from $L^{2}(\mathcal{E}_{h})$ onto $M_{h}$ for $\hat{u}$ on $\Omega_{x}$, and the local $L^{2}$ projections for $u$ and $\bm{q}$,
\begin{equation*}
\begin{aligned}
&-\varepsilon^{-1}(\xi_{q^{1}}, r^{1})_{K_{ij}}+(\xi_{u}, r^{1}_{x})_{K_{ij}}-<\xi_{\hat{u}}(x_{i}^{-}, y), r^{1}(x_{i}^{-}, y)>_{J_{j}}\\
&+<\xi_{\hat{u}}(x_{i-1}^{+}, y), r^{1}(x_{i-1}^{+}, y)>_{J_{j}}=0
\end{aligned}
\end{equation*}
By integrating by parts, it's easy to obtain
\begin{equation}\label{eq:BB-1}
\begin{aligned}
&-\varepsilon^{-1}(\xi_{q^{1}}, r^{1})_{K_{ij}}-((\xi_{u})_{x}, r^{1})_{K_{ij}}+<(\xi_{u}-\xi_{\hat{u}})(x_{i}^{-},y), r^{1}(x_{i}^{-},y)>_{J_{j}}\\
&-<(\xi_{u}-\xi_{\hat{u}})(x_{i-1}^{+},y), r^{1}(x_{i-1}^{+},y)>_{J_{j}}=0
\end{aligned}
\end{equation}
Take $r^{1}|_{K_{ij}}=\frac{x-x_{i-1}}{h_{i, x}}(\xi_{u})_{x}\in \mathcal{Q}^{k}(K_{ij})$, then $r^{1}(x_{i-1}^{+}, y)=0$, $r^{1}(x_{i}^{-}, y)=(\xi_{u})_{x}(x_{i}^{-}, y)$, and
\begin{equation*}
\begin{aligned}
&\left((\xi_{u})_{x}, \frac{x-x_{i-1}}{h_{i, x}}(\xi_{u})_{x}\right)_{K_{ij}}\\&\le C\varepsilon^{-1}\Vert\xi_{q^{1}}\Vert_{K_{ij}}\Vert r^{1}\Vert_{K_{ij}}+C\Vert (\xi_{u}-\xi_{\hat{u}})(x_{i}^{-}, y)\Vert_{J_{j}}\Vert(\xi_{u})_{x}(x_{i}^{-}, y)\Vert_{J_{j}}\\
&\le C\varepsilon^{-1}\Vert\xi_{q^{1}}\Vert_{K_{ij}}\Vert(\xi_{u})_{x}\Vert_{K_{ij}}+Ch_{i, x}^{-\frac{1}{2}}\Vert (\xi_{u}-\xi_{\hat{u}})(x_{i}^{-}, y)\Vert_{J_{j}}\Vert(\xi_{u})_{x}\Vert_{K_{ij}}.
\end{aligned}
\end{equation*}
Here $\Vert(\xi_{u})_{x}(x_{i}^{-}, y)\Vert^{2}_{J_{j}}\le Ch_{j, y}\Vert(\xi_{u})_{x}\Vert^{2}_{L^{\infty}(K_{ij})}\le Ch_{i, x}^{-1}\Vert(\xi_{u})_{x}\Vert^{2}_{K_{ij}}.$
Now we shall prove the following inequality,
$$\Vert(\xi_{u})_{x}\Vert^{2}_{K_{ij}}\le C\left\Vert \left(\frac{x-x_{i-1}}{h_{i, x}}\right)^{\frac{1}{2}}(\xi_{u})_{x}\right\Vert^{2}_{K_{ij}}.$$
From $x_{1}=\frac{2}{h_{i, x}}x+1-\frac{2}{h_{i, x}}x_{i}$, $y_{1}=\frac{2}{h_{j, y}}y+1-\frac{2}{h_{j, y}}y_{j}$ there is
\begin{equation*}
\begin{aligned}
&\int_{K_{ij}}\frac{x-x_{i-1}}{h_{i, x}}(\xi_{u})^{2}_{x}\mr{d}x\mr{d}y\\
&=\int_{-1}^{1}\int_{-1}^{1}\frac{\frac{h_{i, x}}{2}(x_{1}-1+\frac{2}{h_{i, x}}x_{i})-x_{i-1}}{h_{i, x}}\left(\frac{\partial\xi_{u}}{\partial x_{1}}\frac{\partial x_{1}}{\partial x}\right)^{2}\frac{\mr{d}x}{\mr{d}x_{1}}\mr{d}x_{1}\frac{\mr{d}y}{\mr{d}y_{1}}\mr{d}y_{1}\\
&=\int_{-1}^{1}\int_{-1}^{1}\frac{\frac{h_{i, x}}{2}(x_{1}+1)}{h_{i, x}}\left(\frac{\partial\xi_{u}}{\partial x_{1}}\right)^{2}\frac{4}{h_{i, x}^{2}}\frac{h_{i, x}}{2}\frac{h_{j, y}}{2}\mr{d}x_{1}\mr{d}y_{1}\\
&=\int_{-1}^{1}\int_{-1}^{1}\frac{1}{2}(x_{1}+1)(\xi_{u})^{2}_{x_{1}}\frac{h_{j, y}}{h_{i, x}}\mr{d}x_{1}\mr{d}y_{1}\\
&=\frac{1}{2}\frac{h_{j, y}}{h_{i, x}}\int_{-1}^{1}\int_{-1}^{1}(x_{1}+1)(\xi_{u})^{2}_{x_{1}}\mr{d}x_{1}\mr{d}y_{1}.
\end{aligned}
\end{equation*}
Similarly,
\begin{equation*}
\begin{aligned}
\Vert(\xi_{u})_{x}\Vert^{2}_{K_{ij}}&=\int_{K_{ij}}(\xi_{u})_{x}^{2}\mr{d}x\mr{d}y=\int_{-1}^{1}\int_{-1}^{1}(\xi_{u})_{\hat{x}}^{2}\frac{4}{h^{2}_{i, x}}\frac{h_{i, x}}{2}\frac{h_{j, y}}{2}\mr{d}\hat{x}\mr{d}\hat{y}\\
&=\frac{h_{j, y}}{h_{i, x}}\int_{-1}^{1}\int_{-1}^{1}(\xi_{u})^{2}_{\hat{x}}\mr{d}\hat{x}\mr{d}\hat{y}.
\end{aligned}
\end{equation*}
By using the equivalence of the norm on the reference element,
we have

\begin{equation*}
\begin{aligned}
\Vert (\xi_{u})_{x}\Vert^{2}_{K_{ij}}&\le C\left\Vert\left(\frac{x-x_{i-1}}{h_{i, x}}\right)^{\frac{1}{2}}(\xi_{u})_{x}\right\Vert^{2}_{K_{ij}}\\
&\le C\varepsilon^{-1}\Vert \xi_{q^{1}}\Vert_{K_{ij}}||(\xi_{u})_{x}\Vert_{K_{ij}}+Ch_{i, x}^{-\frac{1}{2}}\Vert (\xi_{u}-\xi_{\hat{u}})(x_{i}^{-}, y)\Vert_{J_{j}}\Vert(\xi_{u})_{x}\Vert_{K_{ij}}.
\end{aligned}
\end{equation*}
And then it's easy to obtain
\begin{equation*}
\begin{aligned}
&\Vert (\xi_{u})_{x}\Vert^{2}_{\Omega_{x}}=\sum_{i=N/2+1}^{N}\sum_{j=1}^{N/2}\Vert (\xi_{u})_{x}\Vert^{2}_{K_{ij}}\\
&\le C\varepsilon^{-2}\sum_{i=N/2+1}^{N}\sum_{j=1}^{N/2}\Vert \xi_{q^{1}}\Vert^{2}_{K_{ij}}+Ch_{i, x}^{-1}\sum_{i=N/2+1}^{N}\sum_{j=1}^{N/2}\Vert (\xi_{u}-\xi_{\hat{u}})(x_{i}^{-}, y)\Vert^{2}_{J_{j}}\\
&\le C\varepsilon^{-1}|||\bm{\xi}|||^{2}+C(\tau-\frac{1}{2}\bm{\beta}\cdot \bm{n})^{-\frac{1}{2}}h_{i, x}^{-1}\sum_{i=N/2+1}^{N}\sum_{j=1}^{N/2}\Vert (\tau-\frac{1}{2}\bm{\beta}\cdot \bm{n})^{\frac{1}{2}}(\xi_{u}-\xi_{\hat{u}})(x_{i}^{-}, y)\Vert^{2}_{J_{j}}\\
&\le C\varepsilon^{-1}|||\bm{\xi}|||^{2}+Ch_{i, x}^{-1}|||\bm{\xi}|||^{2}\\
&\le Ch_{i, x}^{-1}|||\bm{\xi}|||^{2}.
\end{aligned}
\end{equation*}
\end{proof}
\vspace{-0.5cm}
\section{Supercloseness}
Through \eqref{eq:coercity} and Lemma \ref{Galerkin orthogonality property}, there is
\begin{equation}\label{eq:uniform-convergence-1}
\begin{split}
|||\bm{\xi}|||^{2} &\le B(\bm{\xi}, \bm{\xi})=B((\bm{q}_{h}-\bm{\Pi_{1}q}, u_{h}-\Pi_{2} u, \hat{u}_{h}-Pu), \bm{\xi})\\
&=B((\bm{q}-\bm{\Pi_{1}q}, u-\Pi_{2} u, \hat{u}-Pu), \bm{\xi})\\ &=B(\bm{\eta}, \bm{\xi})\\
&=\varepsilon^{-1}(\bm{\eta_{q}}, \bm{\xi_{q}})_{\mathcal{T}_{h}}-(\eta_{u}, \nabla\cdot\bm{\xi_{q}})_{\mathcal{T}_{h}}+<\hat{\eta}_{u}, \bm{\xi_{q}}\cdot \bm{n}>_{\partial\mathcal{T}_{h}}\\
&-(\bm{\eta_{q}}, \nabla\xi_{u})_{\mathcal{T}_{h}}-(\bm{\beta}\eta_{u}, \nabla\xi_{u})_{\mathcal{T}_{h}}+((c-\nabla\cdot\bm{\beta})\eta_{u}, \xi_{u})_{\mathcal{T}_{h}}
\\
&+<\bm{\eta_{q}}\cdot \bm{n}, \xi_{u}-\hat{\xi}_{u}>_{\partial\mathcal{T}_{h}}
+<\bm{\beta}\cdot \bm{n}\hat{\eta}_{u}, \xi_{u}-\hat{\xi}_{u}>_{\partial\mathcal{T}_{h}}\\
&+<\tau(\eta_{u}-\hat{\eta}_{u}), \xi_{u}-\hat{\xi}_{u}>_{\partial\mathcal{T}_{h}}\\
&=\varepsilon^{-1}(\bm{\eta_{q}}, \bm{\xi_{q}})_{\mathcal{T}_{h}}-(\eta_{u}, \nabla\cdot\bm{\xi_{q}})_{\mathcal{T}_{h}}-(\bm{\eta_{q}}, \nabla\xi_{u})_{\mathcal{T}_{h}}\\
&-(\bm{\beta}\eta_{u}, \nabla\xi_{u})_{\mathcal{T}_{h}}+((c-\nabla\cdot\bm{\beta})\eta_{u}, \xi_{u})_{\mathcal{T}_{h}}
+<\bm{\eta_{q}}\cdot \bm{n}, \xi_{u}-\hat{\xi}_{u}>_{\partial\mathcal{T}_{h}}\\
&+<\bm{\beta}\cdot \bm{n}\hat{\eta}_{u}, \xi_{u}-\hat{\xi}_{u}>_{\partial\mathcal{T}_{h}}+<\tau\eta_{u}, \xi_{u}-\hat{\xi}_{u}>_{\partial\mathcal{T}_{h}}.
\end{split}
\end{equation}
Here we take advantage of the orthogonality of the local $L^{2}$ projection $P$ and the definition that the stabilization function $\tau$ is taken as a constant.

Because we use local $L^{2}$ projection for $\bm{q}$ and $u$, \eqref{eq:uniform-convergence-1} can be simplified as
\begin{equation*}
\begin{aligned}
|||\bm{\xi}|||^{2} 
&={\large-(\bm{\beta}\eta_{u}, \nabla\xi_{u})_{\mathcal{T}_{h}}}+((c-\nabla\cdot\bm{\beta})\eta_{u}, \xi_{u})_{\mathcal{T}_{h}}
+<\bm{\eta_{q}}\cdot \bm{n}, \xi_{u}-\hat{\xi}_{u}>_{\partial\mathcal{T}_{h}}\\
&+<\bm{\beta}\cdot \bm{n}\hat{\eta}_{u}, \xi_{u}-\hat{\xi}_{u}>_{\partial\mathcal{T}_{h}}+<\tau\eta_{u}, \xi_{u}-\hat{\xi}_{u}>_{\partial\mathcal{T}_{h}}\\
& =:\mathcal{I}_{1}+\mathcal{I}_{2}+\mathcal{I}_{3}+\mathcal{I}_{4}+\mathcal{I}_{5}.
\end{aligned}
\end{equation*}

For $\mathcal{I}_{1}$, we divide the region $\Omega$ into $\Omega_{s}$, $\Omega_{xy}$, $\Omega_{x}$ and $\Omega_{y}$ for analysis. Using the Cauchy-Schwarz inequality and the inverse inequality \cite{cia1:2002-modified}, on $\Omega_{s}\cup\Omega_{xy}$
\begin{equation*}
\begin{aligned}
(\bm{\beta}\eta_{u}, \nabla(\xi_{u}))_{\mathcal{T}_{h}}&=((\bm{\beta}-\bm{P}_{0, h}\bm{\beta})\eta_{u}, \nabla(\xi_{u}))_{\mathcal{T}_{h}}+(\bm{P}_{0, h}\bm{\beta}\eta_{u}, \nabla(\xi_{u}))_{\mathcal{T}_{h}}\\
&=((\bm{\beta}-\bm{P}_{0, h}\bm{\beta})\eta_{u}, \nabla(\xi_{u}))_{\mathcal{T}_{h}}\\
&\le C\left(N^{-1}\Vert\eta_{u}\Vert_{\Omega_{s}}N+(\varepsilon N^{-1}\ln N)\Vert\eta_{u}\Vert_{\Omega_{xy}}(\varepsilon N^{-1}\ln N)^{-1}\right)|||\bm{\xi}|||\\
&\le C\left(N^{-(k+1)}+\varepsilon^{\frac{1}{2}}N^{-(k+1)}(\ln N)^{k+\frac{3}{2}}\right)|||\bm{\xi}|||,
\end{aligned}
\end{equation*}
where $\bm{P}_{0, h}$ is the piecewise-constant projection operator. Then on $\Omega_{x}$ and $\Omega_{y}$, we have
\begin{equation*}
\begin{aligned}
(\bm{\beta}\eta_{u},\nabla(\xi_{u}))_{\mathcal{T}_{h}}&=\sum_{K\in\mathcal{T}_{h}}\int_{K}\eta_{u}\bm{\beta}\cdot\nabla(\xi_{u})\mr{d}x\mr{d}y\\
&=\sum_{K\in\mathcal{T}_{h}}\int_{K}\eta_{u}(\beta_{1}(\xi_{u})_{x}+\beta_{2}(\xi_{u})_{y})\mr{d}x\mr{d}y\\
&=(\beta_{1}\eta_{u}, (\xi_{u})_{x})_{\mathcal{T}_{h}}+(\beta_{2}\eta_{u}, (\xi_{u})_{y})_{\mathcal{T}_{h}}.
\end{aligned}
\end{equation*}
Below we will only analyze the estimate on $\Omega_{x}$, which can be similarly obtained on $\Omega_{y}$. From the Cauchy-Schwarz inequality and  Lemma \ref{L-1},
\begin{equation*}
\begin{aligned}
(\beta_{1}\eta_{u}, (\xi_{u})_{x})_{\mathcal{T}_{h}}
&\le C\Vert\eta_{u}\Vert_{\Omega_{x}}\Vert(\xi_{u})_{x}\Vert_{\Omega_{x}}\\
&\le C\Vert\eta_{u}\Vert_{\Omega_{x}}(\varepsilon N^{-1}\ln N)^{-\frac{1}{2}}|||\bm{\xi}|||\\
&\le C\varepsilon^{\frac{1}{2}}N^{-(k+1)}(\ln N)^{k+\frac{3}{2}}(\varepsilon N^{-1}\ln N)^{-\frac{1}{2}}|||\bm{\xi}|||\\
&\le  CN^{-(k+\frac{1}{2})}(\ln N)^{k+1}|||\bm{\xi}|||,
\end{aligned}
\end{equation*}
and according to the Cauchy-Schwarz inequality and the inverse inequality \cite{cia1:2002-modified},
\begin{equation*}
\begin{aligned}
(\beta_{2}\eta_{u}, (\xi_{u})_{y})_{\mathcal{T}_{h}}&=((\beta_{2}-\bar{\beta_{2}})\eta_{u}, (\xi_{u})_{y})_{\mathcal{T}_{h}}+(\bar{\beta_{2}}\eta_{u}, (\xi_{u})_{y})_{\mathcal{T}_{h}}\\
&\le CN^{-1}\Vert\eta_{u}\Vert_{\Omega_{x}}\Vert(\xi_{u})_{y}\Vert_{\Omega_{x}}\\
&\le CN^{-1}\varepsilon^{\frac{1}{2}}N^{-(k+1)}(\ln N)^{k+\frac{3}{2}}N\Vert\xi_{u}\Vert_{\Omega_{x}}\\
&\le C\varepsilon^{\frac{1}{2}}N^{-(k+1)}(\ln N)^{k+\frac{3}{2}}|||\bm{\xi}|||.
\end{aligned}
\end{equation*}

For $\mathcal{I}_{2}$, from the Cauchy-Schwarz inequality, we have
\begin{equation*}
\begin{aligned}
((c-\nabla\cdot\bm{\beta})\eta_{u}, \xi_{u})_{\mathcal{T}_{h}}&\le C\Vert\eta_{u}\Vert_{\mathcal{T}_{h}}\Vert \xi_{u}\Vert_{\mathcal{T}_{h}}\\
&\le C\left(N^{-(k+1)}+\varepsilon^{\frac{1}{2}}N^{-(k+1)}(\ln N)^{k+\frac{3}{2}}\right)|||\bm{\xi}|||.
\end{aligned}
\end{equation*}

For $\mathcal{I}_{3}$, there is
\begin{equation*}
\begin{aligned}
<\bm{\eta_{q}}\cdot \bm{n}, \xi_{u}-\hat{\xi}_{u}>_{\partial\mathcal{T}_{h}}&\le 
C\left\Vert\left|\tau-\frac{1}{2}\bm{\beta}\cdot \bm{n}\right|^{-\frac{1}{2}}\bm{\eta_{q}}\right\Vert_{\partial\mathcal{T}_{h}}\left\Vert\left|\tau-\frac{1}{2}\bm{\beta}\cdot \bm{n}\right|^{\frac{1}{2}}(\xi_{u}-\hat{\xi}_{u})\right\Vert_{\partial\mathcal{T}_{h}}\\
&\le C\left\Vert\left|\tau-\frac{1}{2}\bm{\beta}\cdot \bm{n}\right|^{-\frac{1}{2}}\bm{\eta_{q}}\right\Vert_{\partial\mathcal{T}_{h}}|||\bm{\xi}|||\\
&\le C\Vert\bm{\eta_{q}}\Vert_{\partial\mathcal{T}_{h}}|||\bm{\xi}|||\\
&\le C\left(\sum_{K\in\partial\mathcal{T}_{h}}\int_{\partial K}(\bm{\eta_{q}}\cdot \bm{n})^{2}\mr{d}s\right)^{\frac{1}{2}}|||\bm{\xi}|||\\
&\le C\left(N^{2}N^{-1}\Vert\bm{\eta_{q}}\Vert^{2}_{L^{\infty}(\Omega)}\right)^{\frac{1}{2}}|||\bm{\xi}|||\\
&\le CN^{-(k+\frac{1}{2})}(\ln N)^{k+1}|||\bm{\xi}|||.
\end{aligned}
\end{equation*}

For $\mathcal{I}_{4}$, from the Cauchy-Schwarz inequality and the assumption on stabilization function $\tau$,
\begin{equation*}
\begin{aligned}
<\bm{\beta}\cdot \bm{n}\hat{\eta}_{u}, \xi_{u}-\hat{\xi}_{u}>_{\partial\mathcal{T}_{h}}&\le C\left\Vert|\bm{\beta}\cdot\bm{n}|^{\frac{1}{2}}\hat{\eta}_{u}\right\Vert_{\partial\mathcal{T}_{h}}\left\Vert|\bm{\beta}\cdot\bm{n}|^{\frac{1}{2}}(\xi_{u}-\hat{\xi}_{u})\right\Vert_{\partial\mathcal{T}_{h}}\\
&\le C\left\Vert|\bm{\beta}\cdot\bm{n}|^{\frac{1}{2}}\hat{\eta}_{u}\right\Vert_{\partial\mathcal{T}_{h}}\left\Vert\left|\tau-\frac{1}{2}\bm{\beta}\cdot \bm{n}\right|^{\frac{1}{2}}(\xi_{u}-\hat{\xi}_{u})\right\Vert_{\partial\mathcal{T}_{h}}\\
&\le C\left(\sum_{K\in\mathcal{T}_{h}}\int_{\partial K}(\hat{\eta}_{u})^{2}\mr{d}s\right)^{\frac{1}{2}}|||\bm{\xi}|||\\
&\le C\left(N^{2}N^{-1}N^{-2(k+1)}(\ln N)^{2(k+1)}\right)^{\frac{1}{2}}|||\bm{\xi}|||\\
&\le CN^{-(k+\frac{1}{2})}(\ln N)^{k+1}|||\bm{\xi}|||.
\end{aligned}
\end{equation*}

For $\mathcal{I}_{5}$, by using the assumption on $\tau$ (3) and (1), we have
\begin{equation*}
\begin{aligned}
<\tau\eta_{u}, \xi_{u}-\hat{\xi}_{u}>_{\partial\mathcal{T}_{h}}
& C\Vert\tau^{\frac{1}{2}}\eta_{u}\Vert_{\partial\mathcal{T}_{h}}\Vert\tau^{\frac{1}{2}}(\xi_{u}-\hat{\xi}_{u})\Vert_{\partial\mathcal{T}_{h}}\\
&\le C\Vert\tau^{\frac{1}{2}}\eta_{u}\Vert_{\partial\mathcal{T}_{h}}\left\Vert\left|\tau-\frac{1}{2}\bm{\beta}\cdot \bm{n}+\frac{1}{2}\bm{\beta}\cdot \bm{n}\right|^{
\frac{1}{2}}(\xi_{u}-\hat{\xi}_{u})\right\Vert_{\partial\mathcal{T}_{h}}\\
&\le C\Vert\tau^{\frac{1}{2}}\eta_{u}\Vert_{\partial\mathcal{T}_{h}}\left\Vert\left|\tau-\frac{1}{2}\bm{\beta}\cdot \bm{n}\right|^{
\frac{1}{2}}(\xi_{u}-\hat{\xi}_{u})\right\Vert_{\partial\mathcal{T}_{h}}\\
&\le C\tau^{\frac{1}{2}}\left(N\Vert\eta_{u}\Vert_{L^{\infty}(\Omega)}^{2}\right)^{\frac{1}{2}}|||\bm{\xi}|||\\
&\le C\tau^{\frac{1}{2}}N^{-(k+\frac{1}{2})}(\ln N)^{k+1}|||\bm{\xi}|||\\
&\le CN^{-(k+\frac{1}{2})}(\ln N)^{k+1}|||\bm{\xi}|||.
\end{aligned}
\end{equation*}

Combining all the estimates above,
\begin{equation}\label{PP-1}
|||\bm{\xi}|||\le CN^{-(k+\frac{1}{2})}(\ln N)^{k+1}.
\end{equation}

Next, we will present two main results of this paper.

\begin{theorem}\label{main}
Suppose that Assumption \ref{ass:S-1} holds and take $\sigma\ge k+1$ on Shishkin mesh, we derive
\begin{align*}
|||(\bm{\Pi_{1}q}-\bm{q}_{h}, \Pi_{2}u-u_{h}, Pu-\hat{u}_{h})|||\le CN^{-(k+\frac{1}{2})}(\ln N)^{k+1},
\end{align*}
where $(\bm{\Pi_{1}q}, \Pi_{2}u, Pu)$ is the projection of the exact solution of  \eqref{eq:S-1} and $(\bm{q}_{h}, u_{h}, \hat{u}_{h})$ is the solution of \eqref{eq:SD}, respectively. 
\end{theorem}

\begin{theorem}\label{main-2}
Let Assumption \ref{ass:S-1} hold and $\sigma\ge k+1$ on Shishkin mesh, one has
\begin{align*}
|||\bm{e}|||\le CN^{-(k+\frac{1}{2})}(\ln N)^{k+1},
\end{align*}
where $\bm{e}:=(\bm{e_{q}}, e_{u}, \hat{e}_{u}):=(\bm{q}-\bm{q}_{h}, u-u_{h}, u-\hat{u}_{h})$, $(\bm{q}, u)$ is the exact solution of  \eqref{eq:S-1} and $(\bm{q}_{h}, u_{h}, \hat{u}_{h})$ is the solution of \eqref{eq:SD}, respectively. 
\end{theorem}
\begin{proof}
From Theorem \ref{error}, Theorem \ref{main} and the triangle inequality, one can derive this conclusion without any difficulties.
\end{proof}

\section{Numerical experiment}
In order to verify the main theoretical conclusion, we consider the following convection-diffusion problem
\begin{equation}\label{eq:KK-2}
\left\{
\begin{aligned}
& -\varepsilon\Delta u+\bm{\beta}\cdot\nabla u+u=f, \quad \text{in $\Omega:= (0,1)^{2}$},\\
& u=0,\quad \text{on $\partial\Omega$},
\end{aligned}
\right.
\end{equation}
where $\bm{\beta}=(2-x,3-y^3)$ and $f(x,y)$ is chosen to satisfy that
\begin{equation*}
u(x,y)=y^3 \sin x (1-e^{-(1-x)/\varepsilon})(1-e^{-2(1-y)/\varepsilon})
\end{equation*}
is the exact solution of the \eqref{eq:KK-2}.

In the experiment, we take $\varepsilon= 10^{-4}, 10^{-5}, \cdots,10^{-8}, N =4, 8, 16, \cdots, 256$ and $k = 1, 2$. 
Furthermore, on Shishkin mesh, set $\beta_1=1$, $\beta_2=2$, $\sigma = k + 1$ and stabilization function $\tau=3$.
Then the corresponding convergence rate is defined as
$$p_{h}= \frac{\ln \bm{e}_{h}-\ln  \bm{e}_{2h}}{(\frac{2\ln N}{\ln 2N})},$$
where
$ \bm{e}_{h}=|||(\bm{q}-\bm{q}_{h}, u-u_{h}, u-\hat{u}_{h})|||$ is the computed error for the particular $\varepsilon$ and $N$. 
Then we will present the numerical results in Table \ref{table:1} and \ref{table:2}, which support Theorem \ref{main-2} sharply.

\begin{table}[H]
\caption{$|||(\bm{q}-\bm{q}_{h}, u-u_{h}, u-\hat{u}_{h})|||$ in the case of $k=1$}
\footnotesize
\resizebox{115mm}{25mm}{
\setlength\tabcolsep{4pt}
\begin{tabular*}{\textwidth}{@{\extracolsep{\fill}} c cccccccccc}
\cline{1-11}
    &\multicolumn{10}{c}{$\varepsilon$ }\\
\cline{1-11}
            \multirow{2}{*}{ $N$ }   &\multicolumn{2}{c}{$10^{-4}$} &\multicolumn{2}{c}{$10^{-5}$}  &\multicolumn{2}{c}{$10^{-6}$}   
&\multicolumn{2}{c}{$10^{-7}$} &\multicolumn{2}{c}{$10^{-8}$}\\

\cline{2-11}&$e_h$&$p_h$&$e_h$&$p_h$&$e_h$&$p_h$&$e_h$&$p_h$&$e_h$&$p_h$\\
\cline{1-11}

             $4$      &1.188e-1  &1.42 &1.188e-1  &1.42 &1.188e-1  &1.42 &1.188e-1  &1.42  &1.188e-1  &1.42\\
             $8$       &6.687e-2  &1.35 &6.687e-2	  &1.35 &6.687e-2  &1.35 &6.687e-2  &1.35  &6.687e-2  &1.35\\
             $16$       &3.553e-2  &1.40  &3.553e-2 &1.40 &3.553e-2  &1.40 &3.553e-2  &1.40  &3.553e-2  &1.40  \\
             $32$       &1.740e-2  &1.44 &1.739e-2  &1.44 & 1.739e-2  &1.44 & 1.739e-2 &1.44  & 1.739e-2 &1.44\\
             $64$       &8.012e-3  &1.46  &8.011e-3	  &1.46 &8.011e-3 &1.46  &8.011e-3  &1.46   &8.011e-3  &1.46   \\
             $128$      &3.540e-3 &1.47 &3.539e-3 &1.47 &3.539e-3  &1.47 &3.539e-3  &1.47 &3.539e-3  &1.47\\
             $256$     &1.519e-3	 &--- &1.519e-3 &---   &1.519e-3 &---   &1.519e-3 &---    &1.519e-3 &---\\

\cline{1-11}
\end{tabular*}}
\label{table:1}
\end{table}

\begin{table}[H]
\caption{$|||(\bm{q}-\bm{q}_{h}, u-u_{h}, u-\hat{u}_{h})|||$ in the case of $k=2$}
\footnotesize
\resizebox{120mm}{25mm}{
\setlength\tabcolsep{4pt}
\begin{tabular*}{\textwidth}{@{\extracolsep{\fill}} c cccccccccccc}
\cline{1-11}
    &\multicolumn{10}{c}{$\varepsilon$ }\\
\cline{1-11}
            \multirow{2}{*}{ $N$ }   &\multicolumn{2}{c}{$10^{-4}$} &\multicolumn{2}{c}{$10^{-5}$}  &\multicolumn{2}{c}{$10^{-6}$}   
&\multicolumn{2}{c}{$10^{-7}$} &\multicolumn{2}{c}{$10^{-8}$} \\

\cline{2-11}&$e_h$&$p_h$&$e_h$&$p_h$&$e_h$&$p_h$&$e_h$&$p_h$&$e_h$&$p_h$\\
\cline{1-11}

             $4$       &2.764e-2  &1.54 &2.764e-2  &1.54 &2.764e-2  &1.54  &2.764e-2  &1.54  &2.764e-2  &1.54  \\
             $8$       &1.482e-2  &1.94 &1.482e-2  &1.94 &1.482e-2  &1.94  &1.482e-2  &1.94   &1.482e-2  &1.94\\
             $16$       &5.966e-3  &2.19 &5.965e-3 &2.19  &5.965e-3 &2.19   &5.965e-3 &2.19  &5.965e-3 &2.19\\
             $32$       &1.945e-3 &2.33  &1.945e-3 &2.33 &1.945e-3 &2.33  &1.945e-3 &2.33    &1.945e-3 &2.33 \\
           $64$       &5.543e-4  &2.40  &5.543e-4  &2.40 &5.543e-4  &2.40  &5.543e-4  &2.40  &5.543e-4  &2.40 \\
             $128$      &1.451e-4  &2.43 &1.451e-4  &2.43 &1.451e-4  &2.43  &1.451e-4  &2.43 &1.451e-4  &2.43     \\
           $256$      &3.587e-5  &--- &3.587e-5  &--- &3.587e-5  &---  &3.587e-5  &---   &3.587e-5  &---     \\

\cline{1-11}
\end{tabular*}}
\label{table:2}
\end{table}


\begin{table}[H]
\caption{$|||(\bm{\Pi_{1}q}-\bm{q}_{h}, \Pi_{2}u-u_{h}, Pu-\hat{u}_{h})|||$ in the case of $k=1$}
\footnotesize
\resizebox{115mm}{25mm}{
\setlength\tabcolsep{4pt}
\begin{tabular*}{\textwidth}{@{\extracolsep{\fill}} c cccccccccc}
\cline{1-11}
    &\multicolumn{10}{c}{$\varepsilon$ }\\
\cline{1-11}
            \multirow{2}{*}{ $N$ }   &\multicolumn{2}{c}{$10^{-4}$} &\multicolumn{2}{c}{$10^{-5}$}  &\multicolumn{2}{c}{$10^{-6}$}   
&\multicolumn{2}{c}{$10^{-7}$} &\multicolumn{2}{c}{$10^{-8}$}\\

\cline{2-11}&$e_h$&$p_h$&$e_h$&$p_h$&$e_h$&$p_h$&$e_h$&$p_h$&$e_h$&$p_h$\\
\cline{1-11}

             $4$      & 2.670e+0  &4.17 &8.442e+0  &4.18 &8.442e+0  &4.18 &8.442e+0  &4.18 &8.442e+0  &4.18\\
             $8$       &4.931e-1  &3.55 &1.549e+0  &3.66 &1.549e+0  &3.66 &1.549e+0  &3.66 &1.549e+0  &3.66\\
             $16$       &9.297e-2  &2.82   &2.770e-1  &3.31  &2.770e-1  &3.31  &2.770e-1  &3.31   &2.770e-1  &3.31  \\
             $32$       &2.204e-2  &1.93 &5.115e-2  &2.80  &5.115e-2  &2.80  &5.115e-2  &2.80 &5.115e-2  &2.80 \\
             $64$       &7.805e-3  &1.52  &1.129e-2  &2.02  &1.129e-2  &2.02  &1.129e-2  &2.02  &1.129e-2  &2.02    \\
             $128$      &3.342e-3 &1.44   &3.639e-3 &1.44  &3.639e-3 &1.44   &3.639e-3 &1.44  &3.639e-3 &1.44  \\
             $256$     &1.455e-3	 &--- &1.455e-3	 &--- &1.455e-3	 &---   &1.455e-3	 &---  &1.455e-3	 &--- \\

\cline{1-11}
\end{tabular*}}
\label{table:3}
\end{table}

\begin{table}[H]
\caption{$|||(\bm{\Pi_{1}q}-\bm{q}_{h}, \Pi_{2}u-u_{h}, Pu-\hat{u}_{h})|||$ in the case of $k=2$}
\footnotesize
\resizebox{120mm}{25mm}{
\setlength\tabcolsep{4pt}
\begin{tabular*}{\textwidth}{@{\extracolsep{\fill}} c cccccccccccc}
\cline{1-11}
    &\multicolumn{10}{c}{$\varepsilon$ }\\
\cline{1-11}
            \multirow{2}{*}{ $N$ }   &\multicolumn{2}{c}{$10^{-4}$} &\multicolumn{2}{c}{$10^{-5}$}  &\multicolumn{2}{c}{$10^{-6}$}   
&\multicolumn{2}{c}{$10^{-7}$} &\multicolumn{2}{c}{$10^{-8}$} \\

\cline{2-11}&$e_h$&$p_h$&$e_h$&$p_h$&$e_h$&$p_h$&$e_h$&$p_h$&$e_h$&$p_h$\\
\cline{1-11}

             $4$       &4.389e-1  &5.86 &4.389e-1  &5.86 &4.389e-1  &5.86 &4.389e-1  &5.86  &4.389e-1  &5.86 \\
             $8$       &4.086e-2  &4.13 &4.086e-2  &4.13 &4.086e-2  &4.13   &4.086e-2  &4.13  &4.086e-2  &4.13 \\
             $16$       &5.857e-3  &2.60&5.857e-3  &2.60  &5.857e-3  &2.60  &5.857e-3  &2.60  &5.857e-3  &2.60 \\
             $32$       &1.553e-3  &2.28  &1.553e-3  &2.28 &1.553e-3  &2.28 &1.553e-3  &2.28    &1.553e-3  &2.28\\
           $64$       &4.534e-4  &2.29  &4.534e-4  &2.29  &4.534e-4  &2.29  &4.534e-4  &2.29  &4.534e-4  &2.29 \\
             $128$      &1.260e-4  &2.34&1.260e-4  &2.34 &1.260e-4  &2.34 &1.260e-4  &2.34&1.260e-4  &2.34\\
           $256$      &3.282e-5  &--- &3.282e-5  &--- &3.282e-5  &---  &3.282e-5 &---   &3.282e-5  &---     \\

\cline{1-11}
\end{tabular*}}
\label{table:4}
\end{table}


\end{CJK}

\end{document}